\documentclass[11pt,a4paper]{article}

\usepackage{amssymb}
\usepackage{amsthm}
\usepackage{amsmath}
\usepackage[colorlinks=true,allcolors=black]{hyperref}
\usepackage{graphics}
\usepackage{enumerate}
\usepackage{tikz}
\usetikzlibrary{positioning}
\usepackage{tkz-tab}
\newtheorem{theorem}{Theorem}[section]
\newtheorem{corollary}[theorem]{Corollary}
\newtheorem{lemma}[theorem]{Lemma}
\newtheorem{proposition}[theorem]{Proposition}
\newtheorem{definition}[theorem]{Definition}

\newtheorem{remark}[theorem]{Remark}

	\newtheorem{question}[theorem]{Question}
 \newtheorem{observation}[theorem]{\bf Observation}

\begin{document}

\title{Forbidden subgraphs on conjugacy class graphs of groups}
 \author{ Papi Ray$^{1}$, Sonakshee Arora$^{2}$, Xuanlong Ma$^{3} $  \\
        \small $^{1}$Indian Institute of Technology, Kanpur-208016, India,\\
       \small  popiroy93@gmail.com.\\
       %\small Orcid iD: 0000-0002-2593-8842\\
       \small $^{2}$Indian Institute of Technology, Jammu-181221, India,\\
        \small tnarora123@gmail.com.\\
        %\small Orcid iD: 0009-0008-4762-0465\\
         \small $^{3}$
School of Science, Xi'an Petroleum University, Xi'an 710065, P. R. China,\\
        \small xuanlma@xsyu.edu.cn\\
        \small Corresponding Author: Xuanlong Ma\\
        }
\date{}
\maketitle
%%%%%%%%%%%%%%%%%%%%%%%%%%%%%%%%%
\
%%%%%%%%%%%%%%%%%%%%%%%%%%%%%%%%%%%%%%%%%%%%%%%%%%%%%%%

\begin{abstract}
Let $G$ be a finite group. The commuting (resp. nilpotent) conjugacy class graph $\Gamma_{CCC}(G)$ (resp. $\Gamma_{NCC}(G)$) of $G$ is a simple graph whose vertex set consists of all non-central conjugacy classes of $G$, in which two distinct vertices $x^G$ and $y^G$ are adjacent if and only if there exist $a \in x^G$ and $b \in y^G$ such that $\langle a, b \rangle$ is an abelian (resp. nilpotent)  subgroup.
In this paper, we mainly investigate cographs,
chordal graphs, split graphs, threshold graphs, and claw-free graphs in terms of forbidden induced subgraphs in $\Gamma_{CCC}(G)$ and $\Gamma_{NCC}(G)$.
To be specific,
we characterize the induced subgraphs in the commuting conjugacy class graph for symmetric groups, alternating groups, and sporadic groups. We also provide a complete classification of these properties for EPPO-groups, nilpotent groups, dihedral groups, dicyclic groups, and generalized dihedral groups in both commuting and nilpotent conjugacy class groups.

\vskip2mm
\noindent \textbf{\textit{keywords}:} {Commuting conjugacy class graph, nilpotent conjugacy class graph, induced
subgraph, finite group.}

\medskip

\end{abstract}

\thispagestyle{empty} \vspace{.4cm}
\section{Introduction}
Graph structures associated with finite groups have been widely studied, utilizing various properties to explore the connections between group theory and graph theory. One notable approach involves defining graphs on certain groups $G$ by considering the non-central conjugacy classes of $G$ as the vertex set, with adjacency determined by different properties of these classes. A detailed survey on conjugacy class graphs can be found in \cite{MR4756899}.

Let $G$ be a finite group. The \textit{commuting conjugacy class graph} of $G$, denoted by $\Gamma_{CCC}(G)$, is a simple graph where the vertices correspond to the non-central conjugacy classes of $G$, and two distinct vertices $x^G$ and $y^G$ are adjacent if and only if there exist $a \in x^G$ and $b \in y^G$ such that $\langle a, b \rangle$ forms an abelian subgroup of $G$. Herzog et al. \cite{MR2561849} first introduced this graph and analyzed its connectivity and established upper bounds on the diameters of its connected components.
Further results on commuting conjugacy class graphs can be found in \cite{MR3567881}, \cite{MR4698474}, and \cite{MR4532813}.

Following this line of research,
Mohammadian and Erfanian \cite{MR3767278} introduced the \textit{nilpotent conjugacy class graph} $\Gamma_{NCC}(G)$, which has the same vertex set as $\Gamma_{CCC}(G)$ and
two distinct vertices $x^G$ and $y^G$ are adjacent if and only if there exist $a \in x^G$ and $b \in y^G$ such that $\langle a, b \rangle$ forms a nilpotent subgroup of $G$. Their results parallel those of \cite{MR2561849}, which provided a characterization of the connected components and determined the diameters of these components.

% Expanding upon these ideas, Bhowal et al. introduced the \textit{solvable conjugacy class graph} $\Gamma_{SCC}(G)$ in \cite{MR4578098}. In this variant, adjacency between two vertices $x^G$ and $y^G$ is defined by the existence of elements $a \in x^G$ and $b \in y^G$ such that $\langle a, b \rangle$ forms a solvable subgroup of $G$. The study in \cite{MR4578098} examined key properties of $\Gamma_{SCC}(G)$, including connectivity, girth, clique number, and other structural characteristics.

Every graph considered in our paper is a simple graph, which is an undirected graph without loops and multiple edges.
Let $\Gamma = (V, E)$ be a graph, where
$V$ and $E$ are
the vertex set and edge set of $\Gamma$, respectively.
Given a subset $F$ of $V$, the {\it induced subgraph} $[F]$ of $\Gamma$ by $F$ is the graph whose vertex set
is $F$ and whose edge set consists of all of the edges in $E$ that have both endpoints in $F$. For any graph $H$, the graph $\Gamma$ is
said to be {\it $H$-free} if it does not contain any induced subgraph isomorphic to $H$. We also say that $H$ is a forbidden induced subgraph of $\Gamma$ if $\Gamma$ is $H$-free.

% \begin{center}Figure $1.1$\end{center}
Forbidden subgraphs characterizations are used to define several important classes of graphs. In the following, we
recall some definitions which are the basis of our investigation.

\begin{definition}
 \begin{enumerate}
     \item[{\rm (1)}] A graph is said to be a \textit{cograph} if it can be constructed from isolated vertices by disjoint union and complementation (see \cite{MR619603}). On the other hand one can see that the \textit{cograph} is the one which is $P_4$-free.
     \item[{\rm (2)}] The \textit{chordal graph} is the one which contains no induced cycles of length greater than three, that is $C_n$-free for any $n\ge 4$ (see \cite{MR232694}).
     \item[{\rm (3)}] A graph is \textit{split} if and only if it contains no induced subgraph isomorphic to
         $C_4, C_5$, or $2K_2$ (see \cite{MR505860}).
     \item[{\rm (4)}] A graph is \textit{threshold graph} if it can be constructed from one vertex graph by repeatedly adding an isolated vertex or dominant vertex (see \cite{MR1417258}). Equivalently, a graph is a threshold graph if and only if it contains no induced subgraph isomorphic to $2K_2$, $C_4$, or $P_4$.

     \item[{\rm (5)}] A graph is called a claw graph if it is isomorphic to the complete bipartite graph $\mathrm{K}_{1,3}$. Equivalently, a graph is claw-free if it contains no vertex with three pairwise non-adjacent neighbours (see \cite{MR4835924}).
 \end{enumerate}
\end{definition}
Every group considered in this paper is finite.
For a group $G$, the {\it commuting graph} of $G$ has vertex set $G$ and two distinct vertices $x$ and $y$ are connected if and only if $xy=yx$ in $G$. Several authors studied various properties of commuting graph.
Morgan and Parker \cite{MR3090056} proved that the connected components of commuting graph when the center is trivial have diameter at most $10$. Furthermore, Kumar et al.  \cite{MR4296337} determined the edge connectivity and the minimum degree of commuting graph and proved that both are equal. They also gave the matching number, clique number and boundary vertex of commuting graph.
Recently, Arvind et al. \cite{MR5080938} discussed the computational problem of deciding whether a given graph is the commuting graph of a finite group and gave new results on the question of whether the commuting graph of a given group is a cograph or a chordal graph.
For more properties of commuting graph, we refer to \cite{MR4730001},\cite{MR4554555},\cite{MR4115328}, and \cite{MR3092687}, and the references therein.

For a finite group $G$, write
$$\mathrm{nil(G)}= \{y \in G \mid \langle x, y \rangle \mbox{ is nilpotent for any } x \in G\},$$
where $\mathrm{nil(G)}$
is called the {\it nilradical} of $G$.
The {\it nilpotent graph} of $G$ has vertex set $G\setminus \mathrm{nil}(G)$, and two distinct vertices $x$ and $y$ are connected if and only if $\langle x,y\rangle$ is a nilpotent subgroup of $G$. Das and Nongsiang \cite{MR3395704} showed that the collection of finite non-nilpotent groups whose nilpotent graphs have the same genus is finite, and derived explicit formulas for the genus of the nilpotent graphs of some well-known classes of finite non-nilpotent groups. Further, they determined all finite non-nilpotent groups whose nilpotent graphs are planar or toroidal.

% Let $G$ be a finite insoluble group. A solvable graph has vertex set $G\setminus \mathrm{R}(G)$ where $\mathrm{R}(G)$ is soluble radical of $G$ and two vertices $x$ and $y$ are connected if $\langle x,y\rangle$ is solvable. In \cite{MR4513813} Burness et al. proved that solvable graph has diameter at most $5$ and proved similar results for almost simple groups. Further in \cite{MR4177578} Akbari et al. proved  conditions on conjugacy classes that imply that the group is solvable. They also showed that the solvable graph is always connected.

Recently, the study of forbidden subgraphs in power graphs has received considerable attention. Manna et al. \cite{MR4281681} characterized all finite nilpotent groups whose power graph is a cograph or a chordal graph. Furthermore, they proved that power graphs are always perfect and classified finite groups whose power graph is a threshold or a split graph.
Later, in \cite{MR4335775}, they determined when the power graph of the direct product of two groups is a cograph.
For simple groups, they showed that in most cases their power graphs are not cographs. Doostabadi et al. \cite{MR3188846}  studied forbidden subgraphs in power graphs by characterizing the finite groups whose power graph is claw-free, $K_{1,4}$-free, or $C_4$-free.
In addition, Ma et al. \cite{MR4746158} classified all finite groups whose enhanced power graph is split and threshold.
Li et al. \cite{MR4863726} classified the finite groups whose trivial intersection power graph is claw-free, $K_{1,4}$-free, $C_4$-free, or $P_4$-free. They also studied finite groups whose trivial intersection power graph is threshold, chordal, or split. Lucchini and Daniele \cite{MR4511156} classified finite groups whose generating graph is a cograph and established that for a finite group $G$, the properties of being split, chordal, and $C_4$-free are equivalent.
Li et al. \cite{MR4236739} provided a complete characterization of finite groups whose reduced power graph is chordal, split, or threshold.

A finite group $G$ is called an {\it EPO-group} if the order of every non-identity element of $G$ has prime order. Similarly, a finite group $G$ is called an {\it EPPO-group} if the order of every non-identity element of $G$ has prime power order. 

In this paper, we discuss the forbidden subgraphs on the commuting and nilpotent conjugacy class graph for various groups. More specifically, in Sect. \ref{sec.4}, we show that both $\Gamma_{CCC}(G)$ and $\Gamma_{NCC}(G)$ are cographs as well as chordal graphs if $G$ is an EPPO-group or a nilpotent group.
We also find necessary and sufficient conditions when the commuting and nilpotent conjugacy class graphs of EPPO-groups  and nilpotent groups are $2K_2$-free and claw-free.
In Sect. \ref{sec.2}, we determine the conditions on $n$ for which the commuting conjugacy class graphs of $\mathrm{Sym}(n)$, and $\mathrm{Alt}(n)$ groups do not contain $P_4$, $2K_2$, and claw as an induced subgraphs.
Sect. \ref{sec.3} extends this analysis for all sporadic groups. 
In Sect. \ref{sec.5}, we obtain the results for dihedral, dicyclic, and generalized dihedral groups in both commuting and nilpotent conjugacy class groups.
In Sect.~\ref{sec.6}, we give some open problems related to the forbidden subgraphs on the conjugacy class graphs of groups.
    % \item For any groups $G$ of order $pq$, where $p$ and $q$ are distinct primes, $\Gamma_{CCC}(G)$, $\Gamma_{NCC}(G)$, and $\Gamma_{SCC}(G)$ are cograph as well as chordal.
    % \item We find the necessary and sufficient conditions when the conjugacy class graph of nilpotent groups are cograph, a chordal graphs, $2K_2$-free, and claw-free.

% In Section \ref{sec.5}, we continue this study for some solvable groups such as dihedral groups, generalized dihedral groups, and dicyclic groups. The main results that we prove in this section are as follows:
% \begin{itemize}
%     \item We find the necessary and sufficient conditions on the order of the groups for which the nilpotent conjugacy class graphs of the above groups is a cograph, a chordal graph, $2K_2$-free, and claw-free.
% \end{itemize}

%%%%%%%%%%%%%%%%%%%%%%%%%%%%%%%%%%%%%%%%%%%%%%%%%%%%%%%%%%%%%%%%

\section{Nilpotent groups and EPPO-groups}\label{sec.4}
 In this section, we will study a few forbidden subgraphs of the commuting conjugacy class graphs of  nilpotent groups and EPPO-groups.
 We start this section by mentioning a proposition that will be required in subsections \ref{sub.nil} and \ref{sub.eppo}.
\begin{proposition}\label{pro.imp}
Let $ G $ be a nilpotent group, and let $ P_1 $ be a non-abelian Sylow $ p $-subgroup for some prime $ p $. Then there exist $ p+1 $ elements in $ P_1 $ that are pairwise non-commuting and not conjugate to each other in $P_1$.
\end{proposition}
\begin{proof}
 Let $a_1\in P_1\setminus Z(P_1)$, where $Z(P_1)$ is the center of $ P_1 $. Let $ C_{P_1}(a_1) = \{ a \in P_1 \mid aa_1 = a_1a \} $ be the centralizer of $ a_1 $ in $ P_1 $. Consider the subgroup $ H := \langle a_1, Z(P_1) \rangle $. Since $ Z(P_1) \subseteq C_{P_1}(a_1) $, it follows that $ H \subseteq C_{P_1}(a_1) $.
Since $ P_1 $ is a $ p $-group, there exists a maximal subgroup $ M_1 $ of $ P_1 $ such that $ C_{P_1}(a_1) \subseteq M_1 $. Choose an element $ a_2 \notin M_1 $, so that $ a_2 \notin C_{P_1}(a_1) $, which implies that $ a_2a_1 \neq a_1a_2 $. Furthermore, since $ a_2 \notin M_1 $, it follows that $ a_2 $ is not conjugate to $ a_1 $.

Next, consider $ C_{P_1}(a_2) $ and a maximal subgroup $ M_2 $ of $ P_1 $ containing $ C_{P_1}(a_2) $. Choose $ a_3 \in P_1 \setminus (M_1 \cup M_2) $, so that $ a_3 $ does not commute with either $ a_1 $ or $ a_2 $, and is also not conjugate to them. By continuing this process iteratively, we can construct the required set of elements.
Moreover,
since $ M_1 $ and $ M_2 $ are normal subgroups of $ P_1 $, their intersection $ M_1 \cap M_2 $ is normal in $ P_1 $. Consider the quotient group $ P_1 / (M_1 \cap M_2) $. This quotient is either isomorphic to $ \mathbb{Z}_{p^2} $ or $ \mathbb{Z}_p \times \mathbb{Z}_p $. Since there exist exactly two subgroups of order $ p $ in $ P_1 / (M_1 \cap M_2) $, it must be isomorphic to $ \mathbb{Z}_p \times \mathbb{Z}_p $.
Now, in $ \mathbb{Z}_p \times \mathbb{Z}_p $, there are precisely $ p+1 $ elements of order $ p $. Their corresponding subgroups in $ P_1 $ will be maximal subgroups containing the centralizers of elements $ a_i \in P_1 $ for $ 1 \leq i \leq p+1 $. Thus, we obtain $ p+1 $ such elements $ a_i $.
\end{proof}

We are now going to state the preliminary lemma, which we use constantly in this subsection.

% \end{lemma}
\begin{lemma}\cite[Lemma 3]{MR2561849}\label{CCC}
 Let $G$ be a locally finite group and $p$ be a prime number. If $x$ and $y$ are $p$-elements then $d(x,y)\le 2$ in $\Gamma_{CCC}(G)$.
\end{lemma}

\subsection{Nilpotent groups}\label{sub.nil}
This subsection aims to analyze the forbidden subgraphs of the conjugacy class graph associated with nilpotent groups. For further details on nilpotent groups, see Chapter 11 of \cite{MR1369573}.

\begin{theorem}
Let $ G $ be a nilpotent group which is not a $ p $-group. The graph $ \Gamma_{CCC}(G) $ is both a cograph and chordal if $ G $ is either an abelian group or a non-abelian group possessing a unique non-abelian Sylow subgroup, with all other Sylow subgroups being abelian, and the commuting conjugacy class graph of this non-abelian Sylow subgroup is both a cograph and a chordal.
\end{theorem}
\begin{proof}
Let $ G $ be a nilpotent group such that
$\Gamma_{CCC}(G)$ is a cograph as well as chordal. If $ G $ is abelian, the statement is immediate. Thus, we assume $ G $ is non-abelian and can be expressed as
$$
G \cong P_{1} \times P_{2} \times \cdots \times P_{r},
$$
where $ P_i $ denotes the Sylow $ p_i $-subgroup of $ G $ for any $ 1 \leq i \leq k $.
We next claim that at most one Sylow subgroup of $ G $ is non-abelian. Suppose, for contradiction, that both $ P_1 $ and $ P_2 $ are non-abelian.
% $$
% a_1 a_2 \neq a_2 a_1 \quad \text{and} \quad b_1 b_2 \neq b_2 b_1.
% $$
Furthermore, by Proposition \ref{pro.imp}, we can choose $ a_1, a_2 \in P_1 $  to be both non-commuting and non-conjugate; a similar argument applies to $ b_1, b_2 \in P_2 $.
Since the elements $ a_i\ \mbox{and}\ b_j $ (for $ 1 \leq i,j\leq 2 $) have distinct orders, it follows that $ a_i $ and $ b_j $ belong to different conjugacy classes. Moreover, as $ a_1 $ and $ a_2 $ are not conjugate, it follows that $ a_2 b_1 $ and $ a_1 b_2 $ are also not conjugate. Consequently, the graph $ \Gamma_{CCC}(G) $ contains the induced path
$$
(a_2b_{1}^{G}, b_{1}^{G}, a_{1}^{G}, a_1b_{2}^{G}),
$$
which contradicts the assumption that $ \Gamma_{CCC}(G) $ is a cograph.
Additionally, $\Gamma_{CCC}(G)$ contains the induced cycle
$$
(a_{1}^{G}, b_{1}^{G}, a_{2}^{G}, b_{2}^{G}, a_{1}^{G}),
$$
contradicting the fact that $ \Gamma_{CCC}(G) $ is chordal.
Thus, we conclude that exactly one Sylow subgroup of $ G $ is non-abelian, while all others are abelian.

Let $ P_1 $ be the unique non-abelian Sylow subgroup of $ G $. Since $ \Gamma_{CCC}(G) $ is both a cograph and chordal, it follows that $ \Gamma_{CCC}(P_1) $ must also be both a cograph and chordal.
\end{proof}

\begin{theorem}
Let $G$ be a nilpotent group. Then $\Gamma_{CCC}(G)$ is $2K_{2}$-free if and only if $G$ is abelian.
\end{theorem}

\begin{proof}
Suppose first that $ \Gamma_{CCC}(G) $ is $ 2K_{2} $-free. Let $ G \cong P_{1} \times P_{2} \times \cdots \times P_{r} $, where $ P_i $ denotes the Sylow $ p_i $-subgroup of $ G $ for each $ 1 \leq i \leq k $. Suppose, for contradiction, that $ G $ is non-abelian. Then at least one Sylow subgroup $ P_{i} $ (for some $ 1 \leq i \leq r $), say $ P_{1} $, is non-abelian.
By Proposition~\ref{pro.imp}, there exist elements $ a_{1}, a_{2} \in P_{1} $ such that they are non-commuting and non-conjugate. Let $ b \in P_{2} $. Since $ b $ belongs to a different Sylow subgroup, it follows that $ a_1, a_2, $ and $ b $ are pairwise non-conjugate. Furthermore, the elements $ a_i $ and $ a_i b $ (for $ i = 1,2 $) have different orders, implying that $ a_i $ and $ a_i b $ are not conjugate either.
As a result, $ \Gamma_{CCC}(G) $ contains the pairs $ \{a_{1}^{G}, {a_{1}b}^{G}\} $ and $ \{a_{2}^{G}, {a_{2}b}^{G}\} $, forming a $ 2K_{2} $, contradicting the assumption that $ \Gamma_{CCC}(G) $ is $ 2K_{2} $-free. Therefore, all Sylow subgroups $ P_{i} $ (for $ 1\leq i \leq r $) must be abelian, implying that $ G $ itself is abelian.

Conversely, if $ G $ is abelian, then $ \Gamma_{CCC}(G) $ is a complete graph, which is trivially $ 2K_{2} $-free.
\end{proof}

From the above lemma, the following corollary is immediate.

\begin{corollary}
Let $G$ be a nilpotent. Then
$\Gamma_{CCC}(G)$ is split, if and only if $\Gamma_{CCC}(G)$ is threshold, if and only if $G$ is an abelian group.
\end{corollary}

\begin{theorem}
Let $G$ be a nilpotent group which is not a $p$-group. Then $\Gamma_{CCC}(G)$ is claw-free if and only if $G$ an abelian group.
\end{theorem}

\begin{proof}
Clearly, if $ G $ is an abelian group,
then $ \Gamma_{CCC}(G) $ is claw-free.
Conversely, suppose that $ \Gamma_{CCC}(G) $ is claw-free. We aim to prove that $ G $ must be abelian. Assume, for contradiction, that $ G $ is non-abelian. Since $ G $ is nilpotent, it has a decomposition of the form
$$
G \cong P_{1} \times P_{2} \times \cdots \times P_{r},
$$
where each $ P_i $ is a Sylow $ p_i $-subgroup of $ G $. As $ G $ is non-abelian, at least one Sylow subgroup $ P_i $ must be non-abelian; let $ P_1 $ be such a subgroup.
By Proposition \ref{pro.imp}, there exist elements $ a_1, a_2, a_3 \in P_1 $ that are pairwise non-commuting and non-conjugate. Consider an element $ b \in P_2 $. Since $ b $ belongs to a different Sylow subgroup, it commutes with each $ a_i $, but the elements $ a_i b $ (for $ i = 1,2,3 $) remain non-conjugate to each other.
As a result, $ \Gamma_{CCC}(G) $ contains an induced claw with central vertex $ b^G $ and three pendant vertices $ (a_1b)^G, (a_2b)^G, $ and $ (a_3b)^G $. This contradicts the assumption that $ \Gamma_{CCC}(G) $ is claw-free.
Therefore, $ G $ is abelian.
\end{proof}
\begin{remark}
For nilpotent group $G$, $\Gamma_{NCC}(G)$ is a complete graph. So it is always a cograph, a chordal, a split as well as a threshold and claw-free.
\end{remark}

%%%%%%%%%%%%%%%%%%%%%%%%%%%%%%%%%%%%%%%%%%%%%%
\subsection{EPPO-groups}\label{sub.eppo}
Recall that a  recall that, a finite group $G$ is called an {\it EPO-group} if the order of every non-identity element of $G$ has prime order  and it is called an
{\it EPPO-group} if the order of every non-identity element of $G$ has prime power order.
For further details on EPPO-groups, see \cite{MR89205}. In this subsection, we primarily examine the forbidden subgraphs of an EPPO-group $ G $ in the graphs
$ \Gamma_{CCC}(G) $ and $ \Gamma_{NCC}(G) $.

We are now going to state a lemma which we use constantly in this subsection.
\begin{lemma}
\label{lm_1}\cite[Lemma 2.1]{MR3767278}
Let $G$ be a locally finite group and $p$ be a prime number.
Then the following statements hold for $\Gamma_{NCC}(G)$ and $\Gamma_{SCC}(G)$:
(i) If $x, y\in  G \setminus {1}$ are $p$-elements, then $d(x, y) \leq 1$.
(ii) If $x, y \in G \setminus {1}$ are of non-coprime orders, then $d(x, y) \leq 3$. Moreover,
$d(x, y) \leq 2$, whenever either $x$ or $y$ is of prime power order.
\end{lemma}

\begin{theorem}\
\label{th_eppo_co}
Let $G$ be an EPPO group. Then both $\Gamma_{CCC}(G)$ and $\Gamma_{NCC}(G)$ are cographs. %[see \cite{ME, Cameron}].
\end{theorem}
\begin{proof}
Suppose that $ x^{G} \sim y^{G} \sim z^{G} \sim w^{G} $ is a path in both $ \Gamma_{CCC}(G) $ and $ \Gamma_{NCC}(G) $. Clearly, the elements of $ G $ have orders that are powers of some prime. If such a path exists, then all the vertices along the path must have orders that are powers of the same prime. Otherwise, $ G $ would contain an element whose order is the product of two distinct primes, contradicting the assumption that $ G $ is an EPPO-group.
Now, by Lemma \ref{CCC}, in $ \Gamma_{CCC}(G) $, the distance between any two elements of the same prime power order is at most $ 2 $. Furthermore, by Lemma \ref{lm_1}, in $ \Gamma_{NCC}(G) $, the distance between any two elements whose orders are powers of the same prime is at most $ 1 $. Therefore, such a path cannot exist in either $ \Gamma_{CCC}(G) $ or $ \Gamma_{NCC}(G) $.
\end{proof}

Note that for an EPPO-group, both of $\Gamma_{CCC}(G)$ and $\Gamma_{NCC}(G)$ are also chordal graphs.

\begin{theorem}\label{lem1}
Let $G$ be an EPPO-group. Then $\Gamma_{CCC}(G)$ is $2K_{2}$-free if and only if one of the following holds:
\begin{enumerate}
\item[{\rm (a)}] $G$ is an abelian $p$-group, where $p$ is a prime;
\item[{\rm (b)}] $G$ is a non-abelian $p$-group with at most one element of order at least $p^2$;
\item[{\rm (c)}] $G$ is a non-abelian and not a $p$-group with Sylow $q$-subgroup having at most one element of order at least $q^2$.
\end{enumerate}
\end{theorem}
\begin{proof}
Let $G$ be an EPPO-group such that $\Gamma_{CCC}(G)$ is $2K_2$-free. We consider the following two cases:  

\medskip
\noindent {\bf Case 1.} $G$ is a $p$-group.
\medskip

If $G$ is abelian, the result is immediate. Now assume $G$ is non-abelian $p$-group. Then by proposition \ref{pro.imp} there exists $a_1,a_2\in G$ which are non-commuting and non-conjugate. If order of $a_1,a_2$ is at least $p^2$, then the pairs $\{a_1,a_1^p\}$ and $\{a_2,a_2^p\}$ will form a  $2K_2$. Hence $G$  can have at most one element of order at least $p^2$.

\medskip
\noindent {\bf Case 2.} $G$ is not a $p$-group.
\medskip

If $G$ is abelian, the result is immediate. Assume $G$ is non-abelian. Let $P$ be a non-abelian Sylow $q$-subgroup of $G$. Since $P$ is a $q$-group so by Case $1$, $P$ can have at most one element of order at least $q^2$.

The proof of the converse is straightforward.
\end{proof}
 %\end{tcolorbox}
\begin{theorem}\label{lem2}
Let $G$ be an EPPO-group. Then $\Gamma_{NCC}(G)$ is $2K_2$-free if and only if one of the following holds:
\begin{enumerate}
    \item[{\rm (a)}] $G$ is abelian;
    \item[{\rm (b)}] $G$ is a $p$-group for some prime $p$;
    \item[{\rm (c)}] $G$ is a non-abelian and not a  $p$-group in which exactly one Sylow subgroup contains elements of prime power order, while all other Sylow subgroups are EPO-groups.
\end{enumerate}
\end{theorem}
\begin{proof}
Suppose first that $\Gamma_{NCC}(G)$ is $2K_2$-free. We next consider three cases: 

\medskip
\noindent {\bf Case 1.} $G$ is cyclic or abelian.
\medskip

If $G$ is either cyclic or abelian, then the nilpotent conjugacy class graph $\Gamma_{NCC}(G)$ is a complete graph. Since a complete graph does not contain $2K_2$ as a subgraph, it follows that $\Gamma_{NCC}(G)$ is $2K_2$-free.

\medskip
\noindent {\bf Case 2.} $G$ is a $p$-group.
\medskip

Since every $p$-group is nilpotent, which ensures that $\Gamma_{NCC}(G)$ is a complete graph. As in the previous case, a complete graph is $2K_2$-free, so the result follows.

\medskip
\noindent {\bf Case 3.} $G$ is a non-abelian group and not a $p$-group.
\medskip

Suppose $G$ has at least two distinct Sylow subgroups, say $P_1$ and $P_2$, each containing elements of prime power order. Let $a \in P_1$ and $b \in P_2$ be such elements. Then, the pairs $\{a^{G}, (a^{p})^{G}\}$ and $\{b^{G}, (b^{p})^{G}\}$ form a $2K_2$ subgraph in $\Gamma_{NCC}(G)$, contradicting our assumption that $\Gamma_{NCC}(G)$ is $2K_2$-free.

Therefore, $G$ must have at most one Sylow subgroup containing elements of prime power order, while all other Sylow subgroups must be EPO-groups, as required.

The proof if the converse is straightforward.
\end{proof}

\begin{theorem}\label{claw2}
Let $G$ be an EPPO-group. Then $\Gamma_{NCC}(G)$ is claw-free. %[see \cite{ME, Cameron}].
\end{theorem}
\begin{proof}
Suppose, for a contradiction, that $\Gamma_{NCC}(G)$ contains a claw whose the three pendant vertices are ${a_1}^{G}, {a_2}^{G}, {a_3}^{G}$ and the central vertex is $b^G$. Then all these elements must have orders that are powers of the same prime. Otherwise, $G$ would contain an element whose order is a product of two primes, which contradicts the assumption that $G$ is an EPPO-group.
However, by Lemma \ref{lm_1}, the distance between any two elements whose orders are powers of the same prime is at most $1$. This means that for $i \neq j$ ($1 \leq i, j \leq 3$), the distance between ${a_i}^{G}$ and ${a_j}^{G}$ in $\Gamma_{NCC}(G)$ is at most $1$. Consequently, such a claw cannot exist in $\Gamma_{NCC}(G)$.
\end{proof}

%%%%%%%%%%%%%%%%%%%%%%%%%%%%%%%%%%%%%%%%%%%%%%%%%%%%%%%%%%%%%%%%%%%%%%%%%%%%%%%%%%%%%%%%%%%%

\section{Forbidden subgraphs of symmetric and alternating groups}\label{sec.2}
We use $\mathrm{Sym}(n)$ and $\mathrm{Alt}(n)$ to denote the symmetric and alternating groups on $n$ letters, respectively.
In this section, we determine the conditions on $n$ under which the commuting conjugacy class graphs of the symmetric and alternating groups contain specific forbidden subgraphs.

\begin{theorem}\label{thm.co}
$\Gamma_{CCC}(\mathrm{Sym}(n))$ is a cograph, split graph, and a threshold graph if and only if $n\leq 4$. Moreover, $\Gamma_{CCC}(\mathrm{Sym}(n))$ is claw-free if and only if $n\leq 5$.
\end{theorem}
\begin{proof}
Let $G=\text{Sym}(n)$.
Consider the elements $ x = (1,2)(3,4,\dots,n) $ and $ y = (1,2,3)(4,\dots,n) $. If $ \Gamma_{CCC}(G) $ is a cograph, then the distance between $ x^G $ and $ y^G $ is at most 2. However, the only neighbours of $ x^G $ are $ (1,2)^G $ and $ (3,\dots,n)^G $, while the only neighbours of $ y^G $ are $ (1,2,3)^G $ and $ (4,\dots,n)^G $ and the sequence  $$
x^G - (1,2)^G - (1,2,3)^G - y^G
$$
forms a path in $ \Gamma_{CCC}(G) $, implying that $ x^G $ and $ y^G $ are in the same connected component of $ \Gamma_{CCC}(G) $.
Note that the above argument does not apply when $ n = 5 $ since, in this case, $ x^G = y^G $. Nevertheless, for $ n = 5 $
$$
(1,2,3,4)^G - (1,3)(2,4)^G - (1,2)^G - (1,2,3)^G.
$$
 we can observe that $ \Gamma_{CCC}(G) $ contains the  $P_4$ as an induced subgraph.

We next show that $\Gamma_{CCC}(G)$ is neither split nor threshold for $n\geq 5$. Let us assume
$$
a = (1,2)(3,4,\dots,n), \quad b = (3,4,\dots,n), \quad c = (1,2,3)(4,\dots,n), \quad d = (4,\dots,n).
$$
Observe that in $ \Gamma_{CCC}(G) $, there exist edges $ a^G - b^G $ and $ c^G - d^G $.  Furthermore, the only neighbours of $ a^G $ are $ (1,2)^G $ and $ (3,\dots,n)^G $, while the only neighbours of $ c^G $ are $ (1,2,3)^G $ and $ (4,\dots,n)^G $. This implies that neither $ a^G $ nor $ b^G $ is adjacent to $ c^G $ or $ d^G $. Consequently, the vertices $ a^G, b^G, c^G, $ and $ d^G $ together form an induced $ 2K_2 $ subgraph.  Note that this argument does not hold when $ n = 5 $, as in this case, $ a^G = c^G $. However, for $ n = 5 $
$$
(1,2,3)^G - (1,2,3)(4,5)^G, \quad(1,2)(3,4)^G - (1,2,3,4)^G.
$$
give an induced $ 2K_2 $ subgraph in $ \Gamma_{CCC}(G) $.

In order to show that $\Gamma_{CCC}(G)$ is claw-free, suppose $n\geq 7$ and
$$a = (1,2), \quad b = (1,2,3), \quad c = (1,2)(3,4,\dots,n), \quad d = (1,2)(3,4)(5,\dots,n).$$
The only neighbours of $ c^G $ in $ \Gamma_{CCC}(G) $ are $ (1,2)^G $ and $ (3,\dots,n)^G $, while the neighbours of $ d^G $ are $ (1,2)^G, (1,2)(3,4)^G, (5,\dots,n)^G $, and $ (1,2)(5,\dots,n)^G $.  Since $ a^G $ is adjacent to $ b^G, c^G $, and $ d^G $, these four vertices form an induced claw in $ \Gamma_{CCC}(G) $, with $ a^G $ as the central vertex and $ b^G, c^G, d^G $ as its pendant vertices. For $ n = 7 $, we observe that $ b^G $ and $ d^G $ are adjacent. However, if we redefine
$$
a = (1,2), \quad b = (1,2,3), \quad c = (1,2)(3,4,\cdots,7), \quad d = (1,2)(3,4)(5,6),
$$
then the graph $ \Gamma_{CCC}(G) $ still contains an induced claw. Again, for $n=6$,
$$
a = (1,2), \quad b = (1,2,3), \quad c = (1,2)(3,4,5,6), \quad d = (1,2)(3,4)(5,6)
$$ will form a claw.

For the converse, we consider the following cases: 

\medskip
\noindent {\bf Case 1.} $n=2,3$.
\medskip 

For $n=2$ and $n=3$, the graph $\Gamma_{CCC}(\mathrm{Sym}(n))$ has too few vertices to contain an induced $P_4$, $2K_2$, or a claw. Hence, $\Gamma_{CCC}(\mathrm{Sym}(n))$ is $P_4$-free, $2K_2$-free, and claw-free.

\medskip
\noindent {\bf Case 2.} $n=4$.
\medskip 

The non-central conjugacy classes
$$
(1,2)^G,\qquad (1,2)(3,4)^G,\qquad (1,2,3,4)^G
$$
form a complete graph $K_3$ in $\Gamma_{CCC}(\mathrm{Sym}(4))$, while the conjugacy class
$(1,2,3)^G$
is an isolated vertex. Consequently, $\Gamma_{CCC}(\mathrm{Sym}(4))$ contains no induced $P_4$, $2K_2$, or claw.

\medskip
\noindent {\bf Case 3.} $n=5$.
\medskip 

In $\Gamma_{CCC}(\mathrm{Sym}(5))$, the conjugacy classes
$$
(1,2)^G,\qquad (1,2)(3,4)^G,\qquad (1,2,3,4)^G
$$
induce a complete graph $K_3$. Moreover, the conjugacy classes
$$
(1,2)^G,\qquad (1,2,3)^G,\qquad (1,2)(3,4,5)^G
$$
also induce a complete graph $K_3$, whereas the conjugacy class $
(1,2,3,4,5)^G
$
is an isolated vertex. Therefore, $\Gamma_{CCC}(\mathrm{Sym}(5))$ contains no induced claw.
\end{proof}

%\begin{theorem}
%$\Gamma_{CCC}(Sym(n))$ is a chordal if and only if $n\leq $.
%\end{theorem}

% \begin{theorem}\label{split.sym}
% $\Gamma_{CCC}(\mathrm{Sym}(n))$ is a split graph as well as a threshold graph if and only if $n\leq 4$.
% \end{theorem}
% \begin{proof}
% Let $ G = \text{Sym}(n) $, and define the elements
% $$
% a = (1,2)(3,4,\dots,n), \quad b = (3,4,\dots,n), \quad c = (1,2,3)(4,\dots,n), \quad d = (4,\dots,n).
% $$
% Observe that in $ \Gamma_{CCC}(G) $, there exist edges $ a^G - b^G $ and $ c^G - d^G $.

% Furthermore, the only neighbours of $ a^G $ are $ (1,2)^G $ and $ (3,\dots,n)^G $, while the only neighbours of $ c^G $ are $ (1,2,3)^G $ and $ (4,\dots,n)^G $. This implies that neither $ a^G $ nor $ b^G $ is adjacent to $ c^G $ or $ d^G $. Consequently, the vertices $ a^G, b^G, c^G, $ and $ d^G $ together form an induced $ 2K_2 $ subgraph.

% Note that this argument does not hold when $ n = 5 $, as in this case, $ a^G = c^G $. However, for $ n = 5 $
% $$
% (1,2,3)^G - (1,2,3)(4,5)^G \quad \text{and} \quad (1,2)(3,4)^G - (1,2,3,4)^G.
% $$
% gives an induced $ 2K_2 $ subgraph in $ \Gamma_{CCC}(G) $.
% \end{proof}

\begin{theorem}\label{co.alt}
  The following holds for $\mathrm{Alt}(n)$:
  \begin{itemize}
\item[{\rm (I)}] $\Gamma_{CCC}(\mathrm{Alt}(n))$ is a cograph if and only if $n\leq 6$;
\item[{\rm (II)}] $\Gamma_{CCC}(\mathrm{Alt}(n))$ is a split graph as well as a threshold graph if and only if $n\leq 5$;
\item[{\rm (III)}] $\Gamma_{CCC}(\mathrm{Alt}(n))$ is claw-free if and only if $n\leq 6$.
\end{itemize}
\end{theorem}
\begin{proof}
\rm(I) First we prove that for $n\geq 7$, $\Gamma_{CCC}(\mathrm{Alt}(n))$ has an induced subgraph isomorphic to $P_4$.

\begin{itemize}
    \item If $ n $ is even and $ n \geq 10 $, consider the elements
    $$
    x = (1,2,3)(4,\dots,n), \quad y = (1,2,3,4,5)(6,\dots,n).
    $$
    \item If $ n $ is odd and $ n \geq 11 $, consider
    $$
    x = (1,2,3)(4,\dots,n-1), \quad y = (1,2,3,4,5)(6,\dots,n-1).
    $$
\end{itemize}

In both cases, it suffices to observe that $ x^G $ and $ y^G $ belong to the same connected component of $ \Gamma_{CCC}(G) $ but have no common neighbours. The corresponding induced path in $ \Gamma_{CCC}(G) $ is
$$
x^G - (1,2,3)^G - (1,2,3,4,5)^G - y^G.
$$

For $n=8$ and $9$, a similar argument applies, with the induced path
$$
(1,2)(3,4)(5,6)(7,8)^G - (1,2)(3,4)^G - (1,2,3,4,5)^G - (1,2,3,4,5)(6,7,8)^G.
$$

If $ n = 7$, then $ \Gamma_{CCC}(G) $ contains the induced subgraph
$$
(1,2,3)(4,5,6)^G - (1,2,3)^G - (4,5)(6,7)^G - (1,2)(4,5,6,7)^G.
$$

(\rm{II}) Next we prove that $\Gamma_{CCC}(G)$ is neither split nor threshold for $n\geq 6$.

\begin{itemize}
    \item If $ n $ is even and $ n \geq 10 $, consider the elements
    $$
    a = (1,2,3)(4,\dots,n), \quad b = (4,\dots,n), \quad c = (1,2,3,4,5)(6,\dots,n), \quad d = (6,\dots,n).
    $$
    Observe that in $ \Gamma_{CCC}(G) $, we have $ \{a^G,b^G\}$  and $\{c^G,d^G\}$ together form a $2K_2$.

    \item If $ n $ is odd and $ n \geq 11 $, consider the elements
    $$
    a = (1,2,3)(4,\dots,n-1), b = (4,\dots,n-1), c = (1,2,3,4,5)(6,\dots,n-1), d = (6,\dots,n-1).
    $$
    By an argument similar to Theorem \ref{thm.co}, we observe that $ \{a^G,b^G\}$ and $\{c^G,d^G\}$ together form a $2K_2$.
\end{itemize}

For $n=8$ and $n=9$, the above argument does not hold, but in these cases, we observe that
$$
(1,2)(3,4)(5,6)(7,8)^G - (1,2)(3,4)^G, \quad (1,2,3)(4,\dots,8)^G - (4,\dots,8)^G
$$
forms a $ 2K_2 $. Suppose $n=7$, consider  $a = (1,\dots,7), \quad b = (1,2)(3,4),$ and $c=(1,2)(3,4,5,6)$. We observe that $a$ and $a^{-1}$ belong to different conjugacy classes. Hence,
$a^G - (a^{-1})^G$ and $b^G - c^G$
form a $ 2K_2 $. For $ n = 6 $, the pairs
$(1,2,3)^G - (1,2,3)(4,5,6)^G,(1,2)(3,4)^G - (1,2)(3,4,5,6)^G$
will form $2K_2$.

%\end{proof}
%\begin{theorem}
%$\Gamma_{CCC}(\mathrm{Alt}(n))$ is a chordal if and only if $n\leq $.
%\end{theorem}
% \begin{theorem}
% $\Gamma_{CCC}(\mathrm{Alt}(n))$ is a split graph as well as a threshold graph if and only if $n\leq 5$.
% \end{theorem}
% \begin{proof}
% Let $ G = \mathrm{Alt}(n) $.

% \begin{itemize}
%     \item If $ n $ is even and $ n \geq 10 $, consider the elements
%     $$
%     a = (1,2,3)(4,\dots,n), \quad b = (4,\dots,n), \quad c = (1,2,3,4,5)(6,\dots,n), \quad d = (6,\dots,n).
%     $$
%     Observe that in $ \Gamma_{CCC}(G) $, we have the edges $ a^G - b^G $ and $ c^G - d^G $.

%     \item If $ n $ is odd and $ n \geq 11 $, take
%     $$
%     a = (1,2,3)(4,\dots,n-1), \quad b = (4,\dots,n-1), \quad c = (1,2,3,4,5)(6,\dots,n-1), \quad d = (6,\dots,n-1).
%     $$
%     By an argument similar to Theorem \ref{split.sym}, the vertices $ a^G - b^G $ and $ c^G - d^G $ together form a $ 2K_2 $.
% \end{itemize}

% For $ n = 8,9 $, the above argument does not hold, but in these cases, we observe that
% $$
% (1,2)(3,4)(5,6)(7,8)^G - (1,2)(3,4)^G \quad \text{and} \quad (1,2,3)(4,\dots,8)^G - (4,\dots,8)^G
% $$
% form a $ 2K_2 $. For $ n = 7 $, consider
% $$
% a = (1,\dots,7), \quad b = (1,2)(3,4), \quad c = (1,2)(3,4,5,6).
% $$
% In $ \mathrm{Alt}(7) $, $ a $ and $ a^{-1} $ belong to different conjugacy classes. Hence,
% $$
% a^G - (a^{-1})^G \quad \text{and} \quad b^G - c^G
% $$
% form a $ 2K_2 $. For $ n = 6 $, the pairs
% $$
% (1,2,3)^G - (1,2,3)(4,5,6)^G \quad \text{and} \quad (1,2)(3,4)^G - (1,2)(3,4,5,6)^G
% $$
% also form a $ 2K_2 $.
% \end{proof}
(\rm{III}) Our next aim is to show that $\Gamma_{CCC}(G)$ is not claw-free for $n\geq 7$.
\begin{itemize}
    \item If $ n $ is even and $ n \geq 10 $, consider the elements
    $$
    a = (1,2,3), \quad b = (1,\dots,5), \quad c = (1,2,3)(4,\dots,n), \quad d = (1,2)(3,4,5)(6,\dots,n-1).
    $$
    \item If $ n $ is odd and $ n \geq 11 $, take
    $$
    a = (1,2,3), \quad b = (1,\dots,5), \quad c = (1,2,3)(4,\dots,n-1), \quad d = (1,2)(3,4,5)(6,\dots,n).
    $$
\end{itemize}
By a similar argument as above, the central vertex $ a^G $ with three pendant vertices $ b^G, c^G, d^G $ will form a claw in $ \Gamma_{CCC}(G) $.

For $ n =7,\ 8,\ 9 $, consider the elements
$$
a = (1,2,3), \quad b = (1,2,3,4,5), \quad c = (1,2)(3,4)(5,6,7), \quad d = (1,2,3)(4,5,6).
$$
Using similar reasoning, the subgraph induced by $ a^G, b^G, c^G, d^G $ will also form a claw.

For the converse, we consider the following cases: 

\medskip
\noindent {\bf Case 1.} $n=3,4$.
\medskip 

For $n=3$ and $n=4$, the graph $\Gamma_{CCC}(\mathrm{Alt}(n))$ does not have sufficient vertices to form an induced $P_4$, $2K_2$, or a claw. Hence, $\Gamma_{CCC}(\mathrm{Alt}(n))$ is $P_4$-free, $2K_2$-free, and claw-free.

\medskip
\noindent {\bf Case 2.} $n=5$.
\medskip 

The non-central conjugacy classes
$$
a=(1,2,3)^G,\qquad b=(1,2)(3,4)^G,\qquad c=(1,2,3,4,5)^G,\qquad c^{-1}=(1,5,4,3,2)^G.
$$
Now, $a$, $b$ are isolated vertices and $c,\ c^{-1}$ form $2K_2$ in $\Gamma_{CCC}(\mathrm{Alt}(5))$. Consequently, $\Gamma_{CCC}(\mathrm{Alt}(5))$ contains no induced $P_4$, $2K_2$, or claw.

\medskip
\noindent {\bf Case 3.} $n=6$.
\medskip 

In $\Gamma_{CCC}(\mathrm{Alt}(6))$, the conjugacy classes
$$
(1,2,3)^G,\qquad (1,2,3)(4,5,6)^G
$$
induce the graph $K_2$. Moreover, the conjugacy classes
$$
(1,2)(3,4)^G,\qquad (1,2)(3,4,5,6)^G,\ \mbox{and}\qquad (1,2,3,4,5)^G,\qquad (1,5,4,3,2)^G
$$
induce the graph $2K_2$. Consequently, $\Gamma_{CCC}(\mathrm{Alt}(6))$ contains no induced $P_4$, or claw.
\end{proof}

\section{Sporadic simple groups}\label{sec.3}
This section is devoted primarily to sporadic simple groups. Recall that there are $26$ sporadic simple groups. Among them, $\mathrm{M}_{11}$ is a subgroup of every sporadic group except
$$
\mathrm{J}_1,\ \mathrm{J}_2,\ \mathrm{J}_3,\ \mathrm{He},\ \mathrm{Ru},\ \mathrm{Th},\ \mbox{and } \mathrm{M}_{22}.
$$
We begin by determining the forbidden subgraphs of the Mathieu group $\mathrm{M}_{11}$. Exploiting the fact that $\mathrm{M}_{11}$ is a subgroup of most sporadic groups, we then determine the forbidden subgraphs for those sporadic groups containing $\mathrm{M}_{11}$. Finally, we determine the forbidden subgraphs for the remaining sporadic groups, namely those that do not contain $\mathrm{M}_{11}$ as a subgroup.
\begin{theorem}\label{thm.M11}
The graph $\Gamma_{CCC}(M_{11})$ is a cograph, a chordal graph, and a claw-free graph. However, it is neither a split graph nor a threshold graph.
\end{theorem}

\begin{proof}
Using GAP, we determine the structure of $\Gamma_{CCC}(M_{11})$ and obtain the following observations.

\begin{itemize}
    \item The induced subgraph on $\{2A,4A,8A,8B\}$ is a complete graph.
    \item The induced subgraph on $\{2A,3A,6A\}$ is also a complete graph.
    \item The vertex $5A$ is isolated.
    \item The vertices $11A$ and $11B$ are adjacent and have no other neighbors.
    \item The edges $\{3A,6A\}$ and $\{11A,11B\}$ induce a copy of $2K_2$.
\end{itemize}
Hence,
$
\Gamma_{CCC}(M_{11})
   \cong K_4 \cup K_3 \cup K_2 \cup K_1,
$
where the two complete subgraphs $K_4$ and $K_3$ intersect at the vertex $2A$.

Since every connected component is complete (or obtained by gluing complete graphs along a cut vertex), the graph contains no induced $P_4$ or induced cycle of length at least $4$. Therefore, $\Gamma_{CCC}(M_{11})$ is both a cograph and a chordal graph. Moreover, every vertex has a neighborhood that is a clique, and hence the graph is claw-free.
Moreover, the induced subgraph on the vertices $\{3A,6A,11A,11B\}$ is isomorphic to $2K_2$. Since $2K_2$ is a forbidden induced subgraph for both split graphs and threshold graphs, it follows that $\Gamma_{CCC}(M_{11})$ is neither a split graph nor a threshold graph.
\end{proof}

Note that $\mbox{M}_{11}$ is a subgroup of the Suzuki groups $G$. It follows that $\Gamma_{CCC}(G)$ is neither split nor threshold (see Corollary \ref{cor.sporadic}).  Now we are going to prove that the commuting conjugacy class graph of the Suzuki groups are always  cographs. Though this does not hold for chordal graphs. Additionally, we demonstrate that the conjugacy class graph of the Suzuki group is not claw-free.
\begin{theorem}
\label{suz_co}
Let $G\cong{}^2B_2(q)$, where $q=2^{2n+1}$. Then $\Gamma_{CCC}(G)$ is cograph.
\end{theorem}

\begin{proof}
 Let $r=2^n$. Then the non-central conjugacy classes of $G$ are:
\begin{itemize}
    \item A single conjugacy class $D_1$ with order of the representative $2$,
    \item Two conjugacy classes $D_2$ and $D_3$ with order of the representatives $4$,
    \item $\frac{(q-2)}{2}$ conjugacy classes $A_1,\cdots, A_{\frac{(q-2)}{2}}$ with order of the representatives $q-1$,
    \item $\frac{(q-2r)}{4}$ conjugacy classes $B_1,\cdots, B_{\frac{(q-2r)}{4}}$ with order of the representatives $q-2r+1$,
    \item $\frac{(q+2r)}{4}$ conjugacy classes $X_1,\cdots, X_{\frac{(q+2r)}{4}}$ with order of the representatives $q+2r+1$.
\end{itemize}
Now, the subgroups of $G$ are:
$C_{2}$, $C_{4}$, $C_{q-1}$, $D_{2(q-1)}$, $C_{q\pm 2r+1}\rtimes H$, where $H$ is a subgroup of order $4$, small order Suzuki groups, $C_{q\pm 2r+1}$, and a non-abelian $2$-group of order $q^2$. Among these, the only abelian subgroups are $C_{2}, C_{4}$, and $C_{q-1}$.
Thus, in $\Gamma_{CCC}(G)$ the only vertices that are connected are $D_1$ with $D_2$ and $D_3$ if $\langle D_1,D_2\rangle\cong C_4$ and $\langle D_1,D_3\rangle \cong C_4$. Otherwise, if $\langle D_1,D_2\rangle\cong C_8$ and $\langle D_1,D_3\rangle\cong C_8$, then $\Gamma_{CCC}(G)$ is null graphs. Thus, in both cases, $\Gamma_{CCC}(G)$ is a cograph.
\end{proof}
% Now suppose $P_1\sim P_2\sim P_3\sim P_4$ is a path in $\Gamma_{SCC}(G)$. Taking $P_1=A_k$ and $P_2=D_1$, we see that the only choices for $P_3$ are $D_2,D_3,X_i,B_j$.

% Now, if we take $P_3=D_2$ or $D_3$, then the only choices for $P_4$ are $X_i$ or $B_j$. But in each of these cases, $P_2$ is connected to $P_4$. Similarly, for $P_3=X_i$ or $B_j$, we get $P_4=D_2$ or $D_3$, but again, in each case, $P_2$ is connected to $P_4$, which contradicts our assumption. All other possible cases also lead to contradictions. Hence, $\Gamma_{SCC}(G)$ is a cograph.

By the similar approach as in
Theorem~\ref{suz_co}, we have the following result.

\begin{theorem}
Let $G \cong {}^2B_2(q)$, where $q = 2^{2n+1}$. Then $\Gamma_{CCC}(G)$ is a chordal graph.
\end{theorem}

\begin{observation}\label{suz.claw}
Let $G \cong {}^2B_2(q)$, where $q = 2^{2n+1}$, then $\Gamma_{CCC}(G)$ is claw-free as it does not have sufficient connected vertices to form a claw.
\end{observation}

\begin{theorem}\label{thm.M12}
The graph $\Gamma_{CCC}(M_{12})$ is neither a cograph nor a chordal graph. Moreover, it is not claw-free.
\end{theorem}
\begin{proof}
The vertices $6A$, $3A$, $3B$, and $6B$ of $\Gamma_{CCC}(M_{12})$ induce a path
    $
    6A-3A-3B-6B,
    $
hence, $\Gamma_{CCC}(M_{12})$ is not a cograph.
The vertices $2A$, $3A$, $3B$, and $2B$ induce the $4$-cycle
    $
    2A-3A-3B-2B-2A,
    $
therefore, $\Gamma_{CCC}(M_{12})$ is not chordal.
The vertex $2A$ is adjacent to each of $4A$, $5A$, and $6A$, while these three vertices are pairwise non-adjacent. Hence, the subgraph induced by
$
\{2A,4A,5A,6A\}
$
is a claw, as desired.
Thus, $\Gamma_{CCC}(M_{12})$ is neither a cograph nor a chordal graph, and it is not claw-free.
\end{proof}

\begin{theorem}\label{thm.M22}
The graph $\Gamma_{CCC}(M_{22})$ is a cograph and a chordal graph, but it is neither a split graph nor a threshold graph. Moreover, it is not claw-free.
\end{theorem}
\begin{proof}
Note that the subgraphs of $\Gamma_{CCC}(M_{22})$ induced by
    \[
    \{2A,3A,6A\},\quad
    \{2A,4A,4B\},\quad
    \{2A,6A,4B\},\quad
    \{2A,4B,8A\}
    \]
are all complete. Now
the vertices $7A$ and $7B$ are adjacent only to each other. Likewise, the vertices $11A$ and $11B$ are adjacent only to each other. Consequently, the edges
$\{7A,7B\}$ and $\{11A,11B\}$
induce a copy of $2K_2$. Hence, $\Gamma_{CCC}(M_{22})$ is neither a split graph nor a threshold graph.
Clearly, the graph contains no induced $P_4$ or induced cycle of length at least four, it follows that it is a cograph as well as a chordal graph.
Furthermore, the vertex $2A$ is adjacent to each of $3A$, $4A$, and $8A$, while these three vertices are pairwise non-adjacent. Hence, the induced subgraph of $\Gamma_{CCC}(M_{22})$ by
$
    \{2A,3A,4A,8A\}
$
is a claw.
Therefore, $\Gamma_{CCC}(M_{22})$ is a cograph and a chordal graph, but it is neither a split graph nor a threshold graph. Moreover, it is not claw-free.
\end{proof}

%%%%%%%%%%%%%%%%%%%%%%%%%%%%%%%%%%%%%%%%%%%%%%%%%%%
%\begin{theorem}\label{thm.M12}
%The graph $\Gamma_{CCC}(M_{12})$ is neither a cograph nor a chordal graph. Moreover, it is not claw-free.
%\end{theorem}
%\begin{proof}
%Analyzing the structure of $\Gamma_{CCC}(M_{12})$ using GAP, we observe the following key properties:

%\begin{itemize}
    %\item The sequence $6A - 3A - 3B - 6B$ forms a path $P_4$.
   % \item The sequence $2A - 3A - 3B - 2B - 2A$ forms a cycle $C_4$.
%\end{itemize}

%These observations confirm that $\Gamma_{CCC}(M_{12})$ is neither a cograph nor a chordal graph.

%Moreover, if we take $2A$ as the central vertex and $\{4A, 5A, 6A\}$ as pendant vertices, then the subgraph induced by $\{2A, 4A, 5A, 6A\}$ forms a claw.
%\end{proof}

%\begin{theorem}\label{thm.M22}
%$\Gamma_{CCC}(M_{22})$ is a cograph and a chordal, but it is not split as well as threshold graph. Moreover it is not claw-free.
%\end{theorem}
%\begin{proof}
%Analyzing the structure of $\Gamma_{CCC}(M_{22})$ using GAP, we observe the following key properties:

%\begin{itemize}
   % \item The subgraphs induced by the sets
   % \item The vertices $7A$ and $7B$ are adjacent only to each other, as are the vertices {$11A$ and $11B$}.
    %\item if we consider {$2A$} as a central vertex and {$3A, 4A,$} and {$8A$} as pendant vertices, then the subgraph induced by $\{2A, 3A, 4A, 8A\}$ forms a claw.
%\end{itemize}

%These observations confirm the result for $\Gamma_{CCC}(M_{12})$.
%\end{proof}
\begin{theorem}\label{cor.sporadic}
For any sporadic simple group $G$,
$\Gamma_{CCC}(G)$ is not a cograph, except for $M_{11}$, $M_{22}$, and the Suzuki groups.
\end{theorem}
\begin{proof}
In $\Gamma_{CCC}(\mbox{Alt}(8))$, there is an induced path
\[
(1,2,3)(4,5,6,7,8)
-
(1,2,3)
-
(4,5)(6,7)
-
(1,2)(4,5,6,7).
\]
Hence, $\Gamma_{CCC}(\mbox{Alt}(8))$ is not a cograph.
Since $\mbox{Alt}(8)$ is a subgroup of the sporadic groups
$$
M_{23},\ Hs,\ Co_{1},\ Co_{2},\ Co_{3},\ Ru,\ \mbox{and}\ M_{24},
$$
the commuting conjugacy class graph of each of these groups contains the corresponding induced subgraph arising from $\mbox{Alt}(8)$.

Indeed, the commutativity relations are preserved under subgroup inclusion. If two conjugacy classes of distinct orders are adjacent in
$\Gamma_{CCC}(\mbox{Alt}(8))$, then the corresponding classes remain adjacent in the ambient sporadic group. Likewise, if they are non-adjacent in $\mbox{Alt}(8)$, they remain non-adjacent in the larger group. Consequently, the induced $P_4$ in
$\Gamma_{CCC}(\mbox{Alt}(8))$ is also an induced subgraph of
$\Gamma_{CCC}(G)$ for each of the above sporadic groups. Therefore, none of
$$
\Gamma_{CCC}(M_{23}),\,
\Gamma_{CCC}(Hs),\,
\Gamma_{CCC}(Co_{1}),\,
\Gamma_{CCC}(Co_{2}),\,
\Gamma_{CCC}(Co_{3}),\,
\Gamma_{CCC}(Ru),\,
\Gamma_{CCC}(M_{24})
$$
is a cograph.

Since $M_{24}$ is a subgroup of $J_{4}$, it follows that
$\Gamma_{CCC}(J_{4})$ is not a cograph. Likewise,
$\mbox{Sym}(10)$ and $\mbox{Sym}(12)$ are subgroups of $Fi_{22}$ and
$Fi_{23}$, respectively, while $Fi_{23}$ is a subgroup of $Fi_{24}'$.
Hence, by Theorem~\ref{thm.co},
$$
\Gamma_{CCC}(Fi_{22}),\quad
\Gamma_{CCC}(Fi_{23}),\quad
\Gamma_{CCC}(Fi_{24}')
$$
are all not cographs. Since $Fi_{23}$ is also a subgroup of $BM$,
$\Gamma_{CCC}(BM)$ is not a cograph.

Furthermore, $\mbox{Alt}(12)$ is a subgroup of both $M$ and $HN$. Therefore,
Theorem~\ref{thm.co} implies that both
$\Gamma_{CCC}(M)$ and $\Gamma_{CCC}(HN)$ are not cographs.

For the remaining sporadic groups, direct GAP computations show that each commuting conjugacy class graph contains an induced path on four vertices:
\begin{itemize}
\item $\Gamma_{CCC}(McL)$: $3B-2A-5A-5B$,
\item $\Gamma_{CCC}(J_1)$: $15B-3A-2A-10B$,
\item $\Gamma_{CCC}(J_2)$: $15A-3A-2A-10C$,
\item $\Gamma_{CCC}(J_3)$: $15B-3A-2A-10B$,
\item $\Gamma_{CCC}(Th)$: $3B-2A-7A-21A$,
\item $\Gamma_{CCC}(He)$: $7C-3A-2A-10A$,
\item $\Gamma_{CCC}(Ly)$: $15C-5B-2A-7A$,
\item $\Gamma_{CCC}(O'N)$: $4B-2A-5A-15A$.
\end{itemize}
Thus, none of these graphs is a cograph.
Finally, by Theorem~\ref{thm.M11}, Theorem~\ref{thm.M22}, and Theorem~\ref{suz_co}, the commuting conjugacy class graphs of $M_{11}$, $M_{22}$, and the Suzuki groups are cographs.
\end{proof}

\begin{theorem}\label{split.sporadic}
For every sporadic simple group $G$, the graph
$\Gamma_{CCC}(G)$ is neither a split graph nor a threshold graph.
\end{theorem}
\begin{proof}
The graph $\Gamma_{CCC}(M_{11})$ contains an induced copy of $2K_2$
formed by the two edges
$\{2A,4A\}$ and $\{3A,6A\}$.
Since the vertices involved have distinct element orders, they remain distinct conjugacy classes in every sporadic group containing a subgroup isomorphic to $M_{11}$. Moreover, commutativity is preserved under subgroup inclusion. Hence every sporadic group containing $M_{11}$ also contains an induced copy of $2K_2$ in its commuting conjugacy class graph. Therefore, these graphs are not split.

The only sporadic groups not containing $M_{11}$ are
$$
J_1,\ J_2,\ J_3,\ He,\ Ru,\ Th,\ \text{and } M_{22}.
$$
Now, $\mbox{Sym}(5)$ is a subgroup of each of
$
He,\ Ru,\ Th,\ \mbox{and } M_{22}.
$
In $\Gamma_{CCC}(\mbox{Sym}(5))$, the edges
$$
(1,2,3)^G-(1,2,3)(4,5)^G
\ \mbox{and }
 (1,2)(3,4)^G-(1,2,3,4)^G
$$
form an induced copy of $2K_2$. Consequently,
$\Gamma_{CCC}(He)$,
$\Gamma_{CCC}(Ru)$,
$\Gamma_{CCC}(Th)$, and
$\Gamma_{CCC}(M_{22})$
are not split.

Finally, using the power-up property (\cite{overleaf_hyperlinks}), we obtain induced copies of $2K_2$ in the remaining three sporadic groups:

\begin{itemize}
    \item In $\Gamma_{CCC}(J_1)$, the vertices
    $15A$, $15B$, $19A$, and $19B$ induce a copy of $2K_2$. Since $\gcd(15,19)=1$ and $J_1$ has no element of order $15\cdot19$, there are no edges joining the two pairs.

    \item In $\Gamma_{CCC}(J_2)$, the vertices
    $3A$, $12A$, $10A$, and $10B$
    induce a copy of $2K_2$.

    \item In $\Gamma_{CCC}(J_3)$, the vertices
    $17A$, $17B$, $19A$, and $19B$
    induce a copy of $2K_2$. Since $\gcd(17,19)=1$ and there is no element of order $17\cdot19$ in $J_3$, the two edges are non-adjacent.
\end{itemize}
Thus every sporadic simple group has a commuting conjugacy class graph containing an induced copy of $2K_2$. Therefore, $\Gamma_{CCC}(G)$ is not a split graph. Since every threshold graph is split, it follows immediately that $\Gamma_{CCC}(G)$ is not a threshold graph.
\end{proof}
%%%%%%%%%%%%%%%%%%%%%%%%%%%%%%%%%%

In Theorem \ref{claw.sporadic} below, we are constantly using \cite{will}, when we are saying about the subgroups of the sporadic groups without mentioning the reference.
\begin{theorem}\label{claw.sporadic}
For any sporadic group $G$, the $\Gamma_{CCC}(G)$ is not claw-free except for $M_{11}\, \mbox{Suzuki}$, and J$_1$.
\end{theorem}
\begin{proof}
From the Theorem \ref{thm.M11}, we know $M_{11}$ is claw-free. Again from Theorem \ref{thm.M12} and Theorem \ref{thm.M22} respectively we know none of $M_{12}$ and $M_{22}$ is claw-free.   The elements in the conjugacy classes which are forming claw are different in orders. Now, $M_{22}$ is a subgroup of both $M_{23}$ and $M_{24}$. So they are also not claw-free. Hence except $M_{11}$ none of the Mathieu groups is claw-free.\par
For both the Conway groups HS and McL, $M_{22}$ is a maximal subgroup. Again, $M_{23}$ is a maximal subgroups for the Conway groups $Co_2$ and $Co_3$. Also, $Co_2$ is a maximal subgroup in the Conway group $Co_1$. So finally we can say, $M_{22}$ is a subgroup in all the Conway groups. Hence none of the Conway groups is claw-free.\par
$\mbox{Sym}(10),\ \mbox{Sym}(12)$ and $\mbox{Alt}(9)$ are the subgroups of the Fischer groups $Fi_{22},\ Fi_{23}$ and $Fi'_{24}$ respectively. Now, from Theorem \ref{thm.co}, we can say none of the $Fi_{22}$ and $Fi_{23}$ is claw-free. Again, if we consider the elements $a=(1,2,3),\ b=(1,2,3,4,5),\ c=(1,2)(3,4)(5,6,7)$ and $d=(1,2,3)(4,5,6,7,8,9)$ in $\mbox{Alt}(9)$, then  $\{a^G,\ b^G,\ c^G,\ d^G\}$ form a claw in $\Gamma_{CCC}(\mbox{Alt}(9))$ with central vertex $a^G$ and pendent vertices $b^G,\ c^G,\ d^G$. Hence $Fi'_{24}$ is also not claw-free. \par
$\mbox{Alt}(12)$ is a subgroup in both the Monster and small monster group HN. Now by Theorem \ref{co.alt}, none of them is claw free. Again, $Fi_{23}$ is a subgroup of the Baby Monster group, so it is also not claw-free.\par
Now from \cite{overleaf_hyperlinks}, using power up property we have the following:
\begin{itemize}
\item In He, $2A$ is some powers of all of $10A$, $12A$, and $14A$ respectively. So $2A$ commutes with all of  $10A$, $12A$, and $14A$. But $10A$, $12A$, and $14A$ do not commute with each other, as there does not have any elements of order $30$ or $35$ or $42$ in He.
\item For the same reasons in Ru, with central vertices $2B$ and pendant vertices $10B,\ 14B$, $26B$ form a claw.
\item In Ly, with central vertices $2A$ and pendant vertices $18A,\ 20A$, $22A$ form a claw.
\item In O'N, with central vertices $2A$ and pendant vertices $10A,\ 12A$, and $14A$ form a claw.
 \item In Th, with central vertices $3A$ and pendant vertices $21A,\ 24A$, and $39A$ form a claw.
\item In J$_2$, with central vertices $2A$ and pendant vertices $8A,\ 10A$, and $12A$ form a claw, where using GAP we checked the conjugacy classes $8A$ and $12A$ do not commute.
\item In J$_3$, with central vertices $2A$ and pendant vertices $8A,\ 10C$, and $12A$ form a claw, where using GAP we checked the conjugacy classes $8A$ and $12A$ do not commute.
\end{itemize}
Again using GAP, we checked J$_1$ does not have any claw. From the Observation \ref{suz.claw}, we know that  Suzuki groups are claw-free.
\end{proof}
%%%%%%%%%%%%%%%%%%%%%%%%%%%%%%%%%%%%%

%%%%%%%%%%%%%%%%%%%%%%%%%%%%%%%%%%%%%%%%%%%%%%

%%%%%%%%%%%%%%%%%%%%%%%%%%%%%%%%%%%%%%%%%%%%%%%%%%%%%%%%%%%%%%%%%%%%%%%%%%%%%%%%%%%%%%%%%%%%
\section{Dihedral groups, dicyclic groups, generalized dihedral groups}\label{sec.5}
In this section, we examine the forbidden subgraphs of the following groups:

\begin{itemize}
    \item The \textbf{dihedral group} $ D_{2n} = \langle x, y \mid x^n = y^2 = 1, \ yxy^{-1} = x^{-1} \rangle $.
    \item The \textbf{dicyclic group} $ T_{4n} = \langle x, y \mid x^{2n} = 1, \ x^n = y^2, \ y^{-1}xy = x^{-1} \rangle $ (see Chapter 12 of \cite{MR2866265} for further details).
    \item The \textbf{generalized dihedral group} $ G = A \rtimes \langle b \rangle $, where $ A $ is an abelian group of order $ n $, and $ \langle b \rangle $ is a cyclic group of order 2 (see Chapter 2 of \cite{MR3643210} for more details).
\end{itemize}
\begin{theorem}\label{thm1}
 Let $ G $ be a finite group. We have the following:
\begin{enumerate}
    \item[{\rm(a)}] If $ G $ is the dihedral group $ D_{2n} $, then $ \Gamma_{CCC}(G) $ is always a cograph and chordal. Moreover, it is a split as well as a threshold graph when $n \not \equiv 2 \pmod{4}$.
    \item[{\rm(b)}] If $ G $ is the dicyclic group $ T_{4n} $, then $ \Gamma_{CCC}(G) $ is always a cograph and {chordal}. Moreover, it is a split as well as a threshold graph when $ n $ is even.
\end{enumerate}
\end{theorem}
\begin{proof}
\noindent (a) From \cite{MR4050270}, we know that the commuting conjugacy class graph of $D_{2n}$ is given by
$$
\Gamma_{CCC}(D_{2n})=
\begin{cases}
K_{\frac{n-1}{2}} \dot{\cup} K_1, & \text{if } n \text{ is odd},\\
K_{\frac{n}{2}-1} \dot{\cup} 2K_1, & \text{if }
n\equiv0\pmod{4} ,\\
K_{\frac{n}{2}-1} \dot{\cup} K_2, & \text{if } n  \equiv 2 \pmod{4},
\end{cases}
$$
where $ \dot{\cup} $ denotes the disjoint union.
It follows from the above structure that $ \Gamma_{CCC}(D_{2n}) $ is always $ P_4 $-free, which implies that it is a {cograph}. Moreover, it is clear that $ \Gamma_{CCC}(D_{2n}) $ is always $ C_n $-free for all $ n \geq 4 $, making it a {chordal graph}. Additionally, when $ n $ is even and $ n/2 $ is odd, the graph $ \Gamma_{CCC}(D_{2n}) $ contains $ 2K_2 $, meaning that in this case, it is neither a {split graph} nor a {threshold graph}.

(b) From \cite{MR4050270}, we know that the commuting conjugacy class graph of $ T_{4n} $ is given by
$$
\Gamma_{CCC}(T_{4n})=
\begin{cases}
K_{n-1} \dot{\cup} 2K_1, & \text{if } n \text{ is even},\\
K_{n-1} \dot{\cup} K_2, & \text{if } n \text{ is odd}.
\end{cases}
$$
From this structure, it is evident that $ \Gamma_{CCC}(T_{4n}) $ is always $ P_4 $-free, making it a {cograph}. Furthermore, it is always $ C_n $-free for all $ n \geq 4 $, ensuring that it is a {chordal graph}. However, when $ n $ is odd, $ \Gamma_{CCC}(T_{4n}) $ contains $ 2K_2 $, indicating that in this case, it is neither a {split graph} nor a {threshold graph}.
\end{proof}
\begin{corollary}
    For all the groups $G$ mentioned in Theorem \ref{thm1}, $\Gamma_{CCC}(G)$ is claw-free.
\end{corollary}

\begin{theorem}
For any integer $n\geq 3$ which is not a power of $2$, we have that $\Gamma_{NCC}(D_{2n})$ is always a cograph and a chordal but  neither a split nor a threshold when $n$ is even but $\frac{n}{2}$ is odd.
\end{theorem}
\begin{proof}
The non-central conjugacy classes of $ D_{2n} $ for odd  $ n $ are given by
$$
[x^i]^{D_{2n}}, 1\leq i\leq \frac{n-1}{2},~~
[y]^{D_{2n}} = \{y, yx, \dots, yx^{\frac{n-1}{2}}\}.
$$
Since the subgroup $ \langle x^i, x^j \rangle $ is always nilpotent, the conjugacy classes $ [x^i]^{D_{2n}} $ (for $ 1\leq i\leq \frac{n-1}{2} $) form a complete graph $ K_{\frac{n-1}{2}} $ in $ \Gamma_{NCC} (D_{2n}) $. Additionally, the subgroup $ \langle x^i, yx^j \rangle $ is isomorphic to $ D_{\frac{2n}{\gcd(n,i)}} $, meaning that $ x^i $ and $ yx^j $ are not adjacent whenever $ \frac{n}{\gcd(n,i)} $ is not a power of 2. Consequently, for odd $ n $, the nilpotent conjugacy class graph is
$$
\Gamma_{NCC} (D_{2n}) = K_{\frac{(n-1)}{2}} \dot{\cup} K_1.
$$
For $ n $ even, the non-central conjugacy classes of $ D_{2n} $ are
$$
[x^i]^{D_{2n}},1\leq i\leq \frac{n}{2}-1,~~
[yx]^{D_{2n}} = \{yx, yx^3, \dots, yx^{\frac{n}{2}-1}\},~~
[y]^{D_{2n}} = \{y, yx^2, \dots, yx^{\frac{n}{2}-2}\}.
$$
Clearly, considering the conjugacy classes $ [x^i]^{D_{2n}} $ (for $ 1\leq i\leq \frac{n}{2}-1 $) as the vertex set, we obtain the complete graph $ K_{\frac{n}{2}-1} $.

If $ \frac{n}{2} $ is also even, then $ yx^{\frac{n}{2}} \in [y]^{D_{2n}} $. Since $ \langle yx, y \rangle \cong D_{2n} $, the vertices $ yx $ and $ y $ are adjacent in $ \Gamma_{NCC} (D_{2n}) $ if and only if $ n $ is a power of 2. Furthermore, $ x^i $ and $ y $ (or $ yx $) are adjacent if and only if the order of $ x^i $ is a power of 2. As a result, in this case, the vertices in $ K_{\frac{n}{2}-1} $ corresponding to elements of order a power of 2 are connected to the vertices in $ 2K_1 $.

If $ \frac{n}{2} $ is odd, then $ yx^{\frac{n}{2}} \in [yx]^{D_{2n}} $, so the subgroup $ \langle yx^i, yx^j \rangle $ has order $ 2^k $ for some $ k $. Thus, $ yx $ and $ y $ are adjacent in $ \Gamma_{NCC} (D_{2n}) $, and in this case,
$$
\Gamma_{NCC} (D_{2n}) = K_{\frac{n}{2}-1} \dot{\cup} K_2.
$$

From the above discussion, it follows that $ \Gamma_{NCC} (D_{2n}) $ is always a cograph and a chordal graph. However, when $ n $ is even and $ \frac{n}{2} $ is odd, $ \Gamma_{NCC} (D_{2n}) $ is neither a split graph nor a threshold graph.
\end{proof}
\begin{remark}
For $n \leq 2$, $\Gamma_{NCC}(D_{2n})$ is a complete graph, so it is always a cograph, a chordal, a split, and a threshold.
\end{remark}

\begin{theorem}
The structure of the nilpotent conjugacy class graph of $ T_{4n} $ is
$$
\Gamma_{NCC}(T_{4n}) = K_{n-1} \dot{\cup} 2K_1, \quad \text{where } n \neq 2^l k.
$$
Consequently, this graph is always a cograph, chordal, threshold, and split, except when $ n = 2^l k $ for $ 0 \leq k < (n-1) $ and $ l \in \mathbb{N} $.
\end{theorem}
\begin{proof}
  The non-central conjugacy classes of $ T_{4n} $ consist of $ [a^k]^{T_{4n}} $ for $ 1 \leq k \leq (n-1) $, along with $ [b]^{T_{4n}} $ and $ [ab]^{T_{4n}} $. Since all elements in the conjugacy class $ [a^k] $ commute with one another, they form a complete subgraph in $ \Gamma_{NCC}(T_{4n}) $.
According to \cite{MR4050270}, the subgroups $ \langle a^k, b \rangle $ and $ \langle a^k, ab \rangle $ are isomorphic to $ T_{\frac{4n}{k}} $, which is not nilpotent unless $ \frac{n}{k} = 2^l $ for some $ l \in \mathbb{N} $. Furthermore, the subgroup $ \langle ab, b \rangle $ is isomorphic to $ T_{4n} $, which is not nilpotent unless $ n $ is power of $2$.
\end{proof}
\begin{theorem}
The nilpotent conjugacy class graph of the generalized dihedral group is always a cograph, chordal, threshold, and split graph, except when $|A| = 2^r$ for some $r \in \mathbb{N}$.
\end{theorem}
\begin{proof}
Let $ A_1 $ denote the subgroup consisting of all elements of order $ 2 $ in $ A $. The non-central conjugacy classes of $ A \rtimes C_2 $ are given by
$$
[b]^G = \{ x_i b x_i^{-1} \mid x_i A_1 \in A/A_1 \} \quad \text{and} \quad [a]^G = \{ a, a^{-1} \mid a \in A \}, \quad 1 \leq i \leq n.
$$
Since $ A $ is abelian, considering the conjugacy classes $ [a]^G $ for $ 1 \leq i \leq n $ as the vertex set forms the complete graph $ K_n $.

For the subgroup $ H= \langle a, x_i b x_i^{-1} \rangle $, since $ x_i b x_i^{-1} $ has order $ 2 $, if $ a $ has order $ m $, then the conjugation relation
$$
(x_i b x_i^{-1})^{-1} a (x_i b x_i^{-1}) = a^{-1}
$$
holds. Consequently, $ H $ is isomorphic to $ \langle a \rangle \rtimes \langle x_i b x_i^{-1} \rangle $. The subgroup $ H $ is nilpotent only when $ m $ is a power of $ 2 $.

Thus, based on this structure, the forbidden subgraphs always form a cograph, a chordal graph, a split graph, and a threshold graph.
\end{proof}
\begin{remark}The commuting conjugacy class graph of a generalized dihedral group is always a cograph, a chordal graph, a threshold graph, and a split graph. This follows from the observation that its structure is given by
\begin{equation*}
    \Gamma_{CCC}(A \rtimes \langle b \rangle) = K_n \dot{\cup} 2K_1.
\end{equation*}
\end{remark}
\begin{corollary}
The nilpotent conjugacy class graph of the dihedral group $ D_{2n} $, the generalized dihedral group, and the dicyclic group $ T_{4n} $ never contains a claw as an induced subgraph.
\end{corollary}

\section{Conclusions and future work}\label{sec.6}
In this paper, we have analyzed the forbidden subgraphs in commuting conjugacy class graphs of symmetric and alternating groups to understand their structural properties. While many intriguing questions remain unexplored.  That leads to the following natural question:

\begin{question}
 For what values of $n$, the nilpotent and solvable conjugacy class graph of the $\mathrm{Sym}(n)$ and the $\mathrm{Alt}(n)$ is cograph, chordal, split, threshold and claw-free?
\end{question}
Moreover, we have studied the forbidden induced subgraphs of all the sporadic groups in commuting conjugacy class graph. A natural question arises in this context.

\begin{question}
Analyze the forbidden subgraphs of nilpotent and solvable conjugacy class graph of the sporadic greoups.
\end{question}

For the finite simple grpops of Lie type the question is totally open. Since for any solvable groups $G$, $\Gamma_{SCC}(G)$ is always complete, so it is always a cograph, a chordal, a split, as well as threshold. Also, it is claw-free. Hence, the natural question is the following.
\begin{question}
  For any solvable groups $G$, what are the conditions on $G$, so that both of $\Gamma_{CCC}(G)$ and $\Gamma_{NCC}(G)$ is a cograph, a chordal, a split, a threshold, and claw-free.
\end{question}

\section*{Acknowledgments}

Dr. Papi Ray expresses gratitude to the Department of Mathematics and Statistics, IIT Kanpur, India for providing funding (in the form of post-doctoral fellowship) during this work.

\section*{Conflict of interest}

The authors declare that there are no conflicts of interest.

%%%%%%%%%%%%%%%%%%%%%%%%%%%%%%%%%%%%%%%%%%%%%%%%%%%%%%%%%%%%%%%%%%%%%%%%%%%%%%%%%%%%%%%%%%%%

%%%%%%%%%%%%%%%%%%%%%%%%%%%%%%%%%%%%%%%%%%%%%%%%%%%%%%%%%%%%%%%%%%%%%%%%%%%%%%GT
\bibliographystyle{plain}
\bibliography{Biblography}

\end{document}